\documentclass[a4paper,12pt]{amsart}
\usepackage[german,english]{babel}
\usepackage{amsmath,amsthm,amssymb,amsfonts}
\usepackage{amscd}
\usepackage[T1]{fontenc}
\usepackage[utf8]{inputenc}
\usepackage[all,cmtip]{xy} 
\usepackage{paralist,url,verbatim, anysize}
\usepackage[usenames, dvipsnames]{color}
\usepackage{graphicx}
\usepackage{hyperref}
\usepackage{enumerate}
\usepackage{euler, amsfonts, amssymb, latexsym, epsfig,epic}

\usepackage{tikz}
\definecolor{darkgrey}{RGB}{100,100,100}
\definecolor{darkblue}{RGB}{0,0,100}
\definecolor{darkgreen}{RGB}{0,100,0}
\definecolor{friendlyred}{RGB}{232,27,0}
\definecolor{mypink}{RGB}{215, 5, 234}

\usepackage{hyperref}
\hypersetup{
  colorlinks=true,
  linkcolor = darkgrey,
  citecolor = darkgreen
}

\usetikzlibrary{decorations.markings,intersections,positioning,calc}

  \tikzset{mylabel/.style  args={at #1 #2  with #3}{
    postaction={decorate,
    decoration={
      markings,
      mark= at position #1
      with  \node [#2] {#3};
 } } } }

\makeatletter
\let\@fnsymbol\@arabic
\makeatother

\theoremstyle{plain}
\newtheorem{theorem}{\bf Theorem}[section]
\newtheorem{conjecture}{Conjecture}
\newtheorem{corollary}[theorem]{Corollary}
\newtheorem{example1}[theorem]{\bf Example}

\newtheorem{lemma}[theorem]{Lemma}

\newtheorem{proposition}[theorem]{Proposition}

\theoremstyle{definition}
\newtheorem{remark}[theorem]{Remark}

\newtheorem{definition}[theorem]{Definition}

\newtheorem*{theorem*}{\bf Theorem}

\theoremstyle{remark}
\newtheorem{claim}{Claim}

\definecolor{mypink}{RGB}{215, 5, 234}

\newcommand{\alF}{F'}

\newcommand{\bfx}{\ensuremath{\mathbf{x}}}

\newcommand{\codim}{\operatorname{codim} }

\newcommand{\CX}{\ensuremath{\mathcal{X}}}

\newcommand{\D}{\Delta}

\newcommand{\init}{\operatorname{in} }
\newcommand{\Jac}{\operatorname{Jac}}

\newcommand{\KK}{\ensuremath{\mathbb{K}}}
\newcommand{\la}{\ensuremath{\lambda}}

\newcommand{\link}{\operatorname{link}}

\newcommand{\mm}{\ensuremath{\mathfrak{m}}}

\newcommand{\NN}{\ensuremath{\mathbb{N}}}
\newcommand{\PP}{\ensuremath{\mathbb{P}}}

\newcommand{\Proj}{\operatorname{Proj} }

\newcommand{\rank}{\operatorname{rank}}

\newcommand{\RR}{\ensuremath{\mathbb{R}}}

\newcommand{\Spec}{\operatorname{Spec} }

\newcommand{\supp}{\operatorname{supp}}

\newcommand{\ZZ}{\ensuremath{\mathbb{Z}}}

\begin{document}
\title{Singularities and radical initial ideals}
\author{Alexandru Constantinescu}
\email{aconstant@zedat.fu-berlin.de}
\address{Institute of Mathematics, Freie Universit\"at Berlin, Germany}
\author{Emanuela De Negri}
\email{denegri@dima.unige.it}
\address{Dipartimento di Matematica, Universit\'a di Genova, Italy}
\author{Matteo Varbaro} 
\email{varbaro@dima.unige.it}
\address{Dipartimento di Matematica, Universit\'a di Genova, Italy} 
  \date{}
\maketitle
\begin{abstract}
What kind of reduced monomial schemes can be obtained as a Gröbner degeneration of a smooth projective variety?
  Our conjectured answer is: only Stanley-Reisner schemes associated to acyclic Cohen-Macaulay simplicial complexes. 
  This would imply, in particular, that only curves of genus zero have such a degeneration.
  We prove this conjecture for degrevlex orders, for elliptic curves over real number fields, for  boundaries of cross-polytopes, and for leafless graphs. We discuss consequences for rational and F-rational singularities of  algebras with straightening laws. 
 \end{abstract}

\section{Introduction}
A deformation of a scheme $X$ is a flat family $\CX\longrightarrow T$, over some connected  affine scheme $T$, whose special fibre is $X$.
Passing from an ideal $I$ of the polynomial ring $\KK[x_1,\dots,x_n]$ to its initial ideal gives rise to a flat family over the parameter space $\Spec \KK[t]$ in which the special fibre (the one over $t=0$) corresponds to the initial ideal. 
We will call such a family a \emph{Gröbner deformation}. 
When the generic fibre of the family is smooth, we call such a deformation a (Gröbner) \emph{smoothing}.
We focus our attention to the situation in which the special fibre is a Stanley-Reisner scheme. This means that $X$ is defined by a square-free monomial ideal $I_\D$ corresponding to some simplicial complex $\D$. 
Specifically, we are looking at the interplay between the type of  singularities (or the lack of such)  in the generic fibre and the topology of  the simplicial complex associated to the special fibre.
We have two reasons for doing this.
On the one hand, having a square-free initial ideal is a desirable property which better preserves homological invariants under degeneration \cite{CV18}. For example, extremal Betti numbers stay constant (in value and position), and thus also depth and regularity, so Cohen-Macaulayness is passed on.
Radical initial ideals appear also in toric settings, where, in certain cases, they  admit a description in terms of unimodular triangulations of lattice polytopes \cite{Stu96,BGT97,BR07,HP09}.
On the other hand, by upper semi-continuity, smoothings (when possible) of  Stanley-Reisner schemes associated to combinatorial manifolds produce important algebraic varieties: spheres smoothen to Calabi-Yau's, tori to Abelian varieties, and triangulations of $\RR\PP^2$ smoothen to Enriques surfaces. Studies in this direction have been done in \cite{AC04,AC10,Chr11,CI14,IT17}.

Throughout  this introduction, the polynomial ring $\KK[x_1,\ldots ,x_n]$ will always be equipped with the standard $\ZZ$-grading given by $\deg x_i=1$. The two reasons  presented above are reflected by two approaches.
In the first one, where the generic fibre is fixed, we start our investigation by merging  two similar questions, \cite[Problem 3.6]{Var18} and \cite[Question 4.2]{CV18}, into one conjecture.
\begin{conjecture}\label{conj1}
Let $I\subset \KK[x_1,\ldots ,x_n]$ be a homogeneous prime ideal defining a nonsingular projective variety such that $\init(I)$ is square-free for some monomial order. Then $\KK[x_1,\ldots ,x_n]/\init(I)$ is Cohen-Macaulay.
\end{conjecture}
Problem 3.6 in \cite{Var18} asks for a counter example of the above conjecture, while Question 4.2 in \cite{CV18} asks if a stronger form of the conjecture is true, namely it asks if the following  holds.
\begin{conjecture}\label{conj2}
Let $I\subset \KK[x_1,\ldots ,x_n]$ be a homogeneous prime ideal defining a nonsingular projective variety such that $\init(I)$ is square-free for some monomial order. Then $\KK[x_1,\ldots ,x_n]/\init(I)$ is Cohen-Macaulay with negative $a$-invariant.
\end{conjecture}

 Conjecture \ref{conj2} is still open for projective curves, in which case it would imply that a nonsingular projective curve with a square-free Gröbner degeneration is forced to have genus zero.

Further motivation  for this approach comes from  algebras with straightening laws (ASL).
These are \KK-algebras whose generators and relations are governed by a finite poset.
All ASL have a discrete counterpart given by a square-free monomial ideal, which can be obtained by a Gröbner degeneration under a degrevlex monomial order.
In particular, all ASL fit our setting.
In \cite{Eis80} and \cite{DEP82} De Concini, Eisenbud and Procesi expressed the feeling that a graded ASL $R$ over a field of characteristic zero should have rational singularities as soon as it is a domain.
It was quickly realised that this cannot be true in full generality, since there exist graded ASL which are nonnormal domains \cite{HW85}.
In unpublished work (mentioned in \cite[p. 11]{DEP82}), Buchweitz proved that  in characteristic zero  an ASL  has rational singularities provided that $R$ has rational singularities on the punctured spectrum and that the discrete part of $R$ (and thus $R$ itself) is Cohen-Macaulay.
 We are able to infer Buchweitz theorem without the hypothesis that the discrete part of $R$ is Cohen-Macaulay, which we prove is a consequence of $R$ having rational singularities on the punctured spectrum (see Corollary \ref{c:asl}).

Conjecture \ref{conj2} trivially holds for hypersurfaces by the following simple, but intriguing, fact. If $f\in \KK[x_1,\ldots ,x_n]$ is a homogeneous polynomial of degree $n$ defining a nonsingular projective hypersurface, then $\init(f)$ cannot be a square-free monomial for any monomial order.
In our second approach, where we fix a Stanley-Reisner ideal, the above fact translates to ``the boundary of a simplex is not Gröbner-smoothable''.  Smoothings do not always exist, but when they exist, how do they look like? Finding explicit equations, even in seemingly simple cases as $n$-cycles,  is  a hard problem. We focus our search to Gröbner smoothings, and start by looking at complete intersections. In this case, smoothings always exist and are obtained by unobstructed generic lifting of first order deformations. It turns out that, just as in the case of hypersurfaces, complete intersections using all the variables are never Gröbner-smoothable (Corollary \ref{cor:ciAreNeverGrobSmoothable}). It is our general feeling that Gröbner smoothings of Stanley-Reisner schemes are in fact rare. For example,  they do not exist for any combinatorial manifold if the order is lexicographic (Proposition \ref{prop:noGroebnerLexSmoothings}).
Another reason for looking at these special deformations is to identify which  deformations can be Gröbner. To our knowledge  there is no general method to determine if a given abstract deformation is Gröbner.
A valuable tool in this approach is provided by Altmann and Christophersen \cite{AC04}, who give complete characterisations of the first two cotangent functors in homological terms related to the simplicial complex. \\

In Proposition \ref{p:equivalent}  we prove that the conjectures \ref{conj1}  and \ref{conj2} are equivalent, by showing that Conjecture \ref{conj1} in dimension $\le d+1$ implies  Conjecture \ref{conj2} in dimension $\le d$. 
For arbitrary monomial orders we know Conjecture \ref{conj1}  for projective curves: this follows by \cite{KS95} if the base field is algebraically closed, and by \cite{Var09} in general (see Proposition \ref{p:necessary}). In Section \ref{sec:degr-reverse-lexic} we show Conjecture \ref{conj2} when the monomial order is a degree reverse lexicographic one (see Corollary \ref{c:degrevlex}).
In the last section we study  Gröbner deformations which are not necessarily associated to degrevlex orders.  Conjectures  \ref{conj1} and \ref{conj2} are equivalent to: A simplicial complex  which is Gröbner-smoothable over $\KK$ must be Cohen-Macaulay and acyclic over $\KK$. We prove the following.
\begin{enumerate}[1.]
\item A combinatorial sphere whose Stanley-Reisner ring is a complete intersection (e.g. the boundary of a cross-polytope) is not Gröbner-smoothable over any field (Corollary \ref{cor:ciAreNeverGrobSmoothable}).
\item A graph with exactly one cycle is not Gröbner-smoothable over any real number field (Theorem \ref{t:elliptic}).
\item A Gröbner-smoothable graph (over some field) has leafs (Corollary \ref{cor:leaflessGraphsNotGRSmoothable}).
\item Any simplicial complex that is Gröbner-smoothable for the lexicographic order has some free face (Proposition \ref{prop:noGroebnerLexSmoothings}).
\end{enumerate}

Theorem~\ref{t:elliptic} is a statement on nonsingular projective curves of genus one, and  its proof  relies on a result from number theory by Elkies \cite{Elk87}. The proofs of the other three results  are essentially  combinatorial.\\

{\it Acknowledgements}: We are grateful to  Klaus Altmann and  Jan Christophersen for the many useful hints on deformation theory; to Anurag Singh for valuable discussions on many topics touched by this paper, especially on the behaviour of the singularities of $\Proj R$ as the graded structure on $R$ varies; to Bhargav Bhatt for clarifications around Theorem \ref{t:elliptic}; to  Mitra Koley for pointing at the generalisation of the result of Fedder and Watanabe crucial for the proof of the positive characteristic case of Corollary \ref{c:degrevlex}; to Mateusz Michalek for many experiments on Theorem  \ref{t:elliptic}.
\vskip 1cm

\section{The conjectures}
\label{sec:conjectures}

Let $S$ be the  polynomial ring $\KK[x_1,\dots,x_n]$ over a field $\KK$ equipped with a positive grading (i.e. $\deg x_i>0$ for all $i=1,\ldots ,n$). For a homogeneous ideal $I\subset S$, let $R=S/I$ and denote by $H_\mm^i(R)$ the $i$th local cohomology module of $R$ supported at the unique homogeneous maximal ideal $\mm$ of $R$. The $S$-module $H_\mm^i(R)$ is $\ZZ$-graded; for $j\in\ZZ$, we will denote by $H_\mm^i(R)_j$ its degree $j$ part. 

We recall that $R$ is \emph{Cohen-Macaulay} if and only if $H_\mm^i(R)=0$ for all $i<\dim R$.
The {\it $a$-invariant} of a Cohen-Macaulay positively graded \KK-algebra $R$ is the maximum $j\in\ZZ$ such that $H_\mm^{\dim R}(R)_j\neq 0$.
In  this case, the $a$-invariant of $R$ is negative if and only if $H^{\dim X}(X,\mathcal{O}_X)=0$ (where $X=\Proj R$), so  having negative $a$-invariant depends only on $X$.
The ring $R$ is called {\it generalised Cohen-Macaulay} if and only if $H_\mm^i(R)$ has finite length for any $i<\dim R$.
In positive characteristic, $R$ is called {\it $F$-injective} if the natural Frobenius action on $H^i_{\mm}(R)$ is injective for all $i\in \NN$. For the definitions of rational singularities and $F$-rational singularities we refer to \cite[Chapter 10]{BH93}. 

\begin{remark}
  \label{rem:gCMandPuncturedSpectrum}
A graded, finitely generated $R$-module $M$ has finite length if and only if $\mathrm{Supp}M=\{\mm\}$. Using graded duality, it is not difficult to realise that $R$ is generalised Cohen-Macaulay if and only if $R$ is equidimensional and Cohen-Macaulay on the punctured spectrum $\Spec R\setminus \{\mm\}$. The latter condition is implied by $R$ being an equidimensional isolated singularity, or, in characteristic zero, by $R$ being equidimensional and having rational singularities on the punctured spectrum.
\end{remark}
\begin{remark}
  \label{rem:standardVsNonstandardGrading}
  In the standard graded situation, $R$ is an isolated singularity if and only if $X=\Proj R$ is a nonsingular projective scheme.
In the nonstandard graded polynomial ring one has to be careful:
  For instance, $R$ being an isolated singularity is not equivalent to $\Proj R$ being nonsingular
  (e.g. the weighted projective space $\PP(1,1,2)$ is singular).
\end{remark}

For convenience,  we recall the results of \cite{CV18} which we will most frequently use.

\begin{theorem}[\cite{CV18}]
\label{CV}
Let $I\subset S$ be a homogeneous ideal such that $\init(I)$ is radical for some monomial order. Then
\begin{enumerate}[\upshape (i)]
\item $\dim_K H^i_{\mm}(S/I)_j=\dim_K H^i_{\mm}(S/\init(I))_j$ for all $i,j\in\ZZ$.
\item $S/I$ is Cohen-Macaulay if and only if $S/\init(I)$ is Cohen-Macaulay. 
\item $S/I$ is generalised Cohen-Macaulay if and only if $S/\init(I)$ is Buchsbaum. 
\item $S/I$ satisfies Serre condition $(S_r)$ if and only if $S/\init(I)$ satisfies Serre condition $(S_r)$, for all $r\geq 1$. 
\end{enumerate}
\end{theorem}

All the above results, even if originally stated for the standard grading, hold true for any positive grading \cite[Remark 2.5]{CV18}.

For the remaining part of this section, the graded structure on $S=\KK[x_1,\ldots ,x_n]$ will be the standard one.

\begin{proposition}\label{p:equivalent}
Conjectures \ref{conj1} and \ref{conj2} are equivalent.
\end{proposition}
\begin{proof}
  
  Obviously \ref{conj2} implies \ref{conj1}, so let us assume \ref{conj1} is true and take a homogeneous prime ideal $I\subset S=\KK[x_1,\ldots ,x_n]$ contradicting \ref{conj2}.
  We have that $X=\Proj S/I$ is nonsingular, $\init(I)$ is square-free for some monomial order, $S/\init(I)$ is Cohen-Macaulay, but the $a$-invariant of $S/\init(I)$ is nonnegative.
  So also $S/I$ is Cohen-Macaulay with  nonnegative $a$-invariant.
  We show now that the Segre embedding $Y=X\times \PP^1\subset \PP^{2n-1}$ provides a counterexample to Conjecture \ref{conj1}: Let $A=\KK[y_{ij}:i\in \{0,1\},j\in\{1,\ldots ,n\}]$ and $P\subset A$ be the homogeneous prime ideal defining  $Y\subset \PP^{2n-1}$.
  Then $\Proj A/P\cong Y$ is nonsingular and $H_M^d(A/P)_0\cong H_{\mm}^d(S/I)_0\neq 0$ by \cite[Theorem 4.1.5]{GW78}, where $M$ is the unique homogeneous maximal ideal of $A/P$ and $d=\dim S/I$.
  Since $\dim A/P=d+1$, the ring $A/P$ is not Cohen-Macaulay, so $A/\init(P)$ is not Cohen-Macaulay for any monomial order on $A$.
 Finally, $\init(P)$ is square-free for some special monomial order by \cite[Theorem 3.2]{Shi12}.
\end{proof}

\begin{remark}
In the proof of Proposition~\ref{p:equivalent}  we actually showed that
\[
\mbox{Conjecture \ref{conj1} in dimension $\leq d+1$}\implies \mbox{Conjecture \ref{conj2} in dimension $\leq d$}.
\]
\end{remark}

Let us recall that there is a one-to-one correspondence between square-free monomial ideals $I_\D$ of $S$ and simplicial complexes $\D$ on $[n]=\{1,\ldots ,n\}$, given by:
\[\forall \ F\subset [n], \ F\in\D\iff \bfx:=\prod_{i\in F}x_i\notin I_{\D}.\]
In some instances it will be convenient to identify $F$ and $\bfx_F$, and simply write  $\bfx_F\in\D$.
An element of $\D$ is called {\it face}. By {\it maximal face} we mean maximal under inclusion. 
The dimension of a face $F\in \D$ is $\dim F=|F|-1$, and the dimension of $\D$ is
\[\dim \D=\max\{\dim F:F\in\D\}.\]
 The {\it link} of  $F\in \D$ is the simplicial complex $\link_\D F=\{G\subset [n]: G\cup F\in\D \mbox{ and }G\cap F=\emptyset\}$.
A simplicial complex $\D$ is {\it pure} precisely when all its maximal faces  have the same dimension. A simplicial complex is {\it strongly connected} if and only if for every pair of maximal faces $F,G\in\D$, there exist $F=F_0,F_1,\ldots ,F_r=G$ maximal faces of $\D$ such that $|F_i\cap F_{i+1}|=|F_i|-1$ for all $0\leq i<r$; in particular, if $\D$ is strongly connected then it is pure. We say that $\D$ is {\it normal} if and only if $\link_\D F$ is strongly connected for all $F\in\D$. Furthermore, $\D$ is \emph{Cohen-Macaulay} if $\widetilde{H}^i(\link_\D F;\KK)=0$ for all $F\in \D$ and $i<\dim \link_\D F$, while $\D$ is \emph{Buchsbaum} if $\widetilde{H}^i(\link_\D F;\KK)=0$ for all $\emptyset \neq F\in \D$ and $i<\dim \link_\D F$. These two notions agree with the corresponding algebraic notions for the Stanley-Reisner ring $\KK[\D]=S/I_{\D}$. Finally, if $\KK[\D]$ is Cohen-Macaulay, then it has negative $a$-invariant if and only if $\widetilde{H}^{\dim \D}(\D;\KK)=0$.

\begin{proposition}\label{p:necessary}
Let $I\subset S$ be an ideal such that $\init(I)=I_{\D}$ for some monomial order.
\begin{enumerate}[\upshape (i)]
\item If $S/I$ is a domain, then $\D$ is strongly connected.
\item If $S/I$ is a normal domain, then $\D$ is normal.
\item If $I$ is homogeneous and $\Proj S/I$ is an equidimensional nonsingular projective scheme, then $\D$ is Buchsbaum.
\end{enumerate}
\end{proposition}  
 \begin{proof}
   Point (i) follows by \cite{KS95} if the base field is algebraically closed, and by \cite{Var09} in general. \\
   (ii) Since $R/I$ is normal, it satisfies Serre condition $(S_2)$, so, by Theorem \ref{CV} (iv), $\KK[\D]$ satisfies $(S_2)$ as well; on the other hand, it is known and easy to prove that $\D$ is normal if and only if $\KK[\D]$ satisfies $(S_2)$. \\
   (iii) As $S/I$ is generalised Cohen-Macaulay,  we conclude by Theorem \ref{CV} (iii).
\end{proof}
 
Given a simplicial complex $\Delta$ on $n$ vertices, we say it is {\it Gröbner-smoothable} over $\KK$ if there exists a homogeneous prime ideal $I\subset \KK[x_1,\ldots ,x_n]$ defining a nonsingular projective variety and a monomial order such that $\init(I)=I_{\Delta}$. So a further equivalent formulation of Conjectures \ref{conj1} and \ref{conj2}, using Proposition \ref{p:equivalent}, is the following:
 
 \begin{conjecture}\label{conj3}
   A simplicial complex which is Gröbner-smoothable over $\KK$ is Cohen-Macaulay and acyclic over $\KK$.
 \end{conjecture}
 A simplicial complex $\D$ is \emph{acyclic over $\KK$} if all its  reduced cohomology with coefficients in $\KK$ vanishes.

\section{Degree reverse lexicographic degenerations}
\label{sec:degr-reverse-lexic}

Throughout this section $S=\KK[x_1,\ldots ,x_n]$ is again allowed to have any positive grading.

\begin{definition}
  \label{def:xnaddictedOrder}
A monomial order on $S$ is {\it $x_n$-addicted} if, whenever $f\in S$ and $x_n|\init (f)$, we have that $x_n|f$.
\end{definition}

\begin{remark}
  \label{everyDRLisAddicted}
Any degrevlex monomial order is $x_n$-addicted if $x_n$ is the smallest variable.
\end{remark}

\begin{theorem}
  \label{t:xnaddicted}
  Let $I\subset S$ be a homogeneous ideal such that $S/I$ is generalised Cohen-Macaulay and $x_n$ is $S/I$-regular. If $\init (I)$ is radical for some $x_n$-addicted monomial order, then $S/\init (I)$ is Cohen-Macaulay with negative $a$-invariant.
\end{theorem}

\begin{proof}
  By Theorem \ref{CV} (iii), $S/\init(I)$ is generalised Cohen-Macaulay.
  By Hochster's formula for local cohomology \cite[Theorem 3.5.8]{BH93}  we infer:
  \[
    H_\mm^i(S/\init(I))=H_\mm^i(S/\init(I))_0\cong \widetilde{H}^{i-1}(\Delta) \ \ \forall \ i\geq 1,
  \] 
where $\D$ is the simplicial complex on $n$ vertices such that $\init(I)=I_\D$, the Stanley-Reisner ideal of $\Delta$.
Since $H_\mm^0(S/\init(I))$ trivially vanishes, it is enough to show that $\widetilde{H^i}(\Delta)=0$ for all $i\in\NN$.
We will show an even stronger statement, namely, that $\Delta$ is contractible.

In order to do so, it is enough to show that the vertex corresponding to $x_n$ is a cone point for $\Delta$, i.e. that $x_n$ is contained in every maximal face of $\D$.
This is equivalent to  $x_n$  not dividing any minimal monomial generator of $\init(I)$.
By contradiction, let $u$ be a minimal monomial generator of $\init(I)$ which is divisible by $x_n$, and let $f\in I$ such that $\init(f)=u$.
As the monomial order  is $x_n$-addicted, there exists $g\in S$ such that $f=x_ng$. Since $x_n$ is $S/I$-regular, we have $g\in I$. In particular,  $\init(g)$ is a monomial of $\init(I)$ strictly dividing $u$, thus contradicting the minimality of $u$.
\end{proof}

The following corollary solves  Conjectures \ref{conj1} and \ref{conj2} for any degree reverse lexicographic monomial order. 

\begin{corollary}\label{c:degrevlex}
Let $I\subset S$ be homogeneous prime ideal such that $\init (I)$ is a square-free monomial ideal with respect to a degrevlex monomial order. 
\begin{enumerate}[\upshape (i)]
\item If $S/I$ is generalised Cohen-Macaulay, then $S/\init(I)$ (and thus also $S/I$) is Cohen-Macaulay with negative $a$-invariant.
\item In characteristic zero, $S/I$ has rational singularities whenever it has rational singularities on the punctured spectrum.
\item In positive characteristic, $S/I$ is $F$-rational whenever it is $F$-rational on the punctured spectrum.
\end{enumerate}
\end{corollary}
\begin{proof}
(i) We can assume that $x_1>\ldots >x_n$. If $x_n\notin I$, then we conclude by Theorem \ref{t:xnaddicted}. Otherwise, set $S'=S/(x_n)$ and $I'=I/(x_n)$. Then $I'\subset S'$ is a prime ideal such that $S'/I'\cong S/I$ is generalised Cohen-Macaulay and $\init (I')=\init(I)/(x_n)$ is a square-free monomial ideal with respect to the degrevlex monomial order extending the linear order $x_1>\ldots >x_{n-1}$. By induction on the number of variables, $S'/\init(I')\cong S/\init(I)$ is Cohen-Macaulay with negative $a$-invariant.

(ii) It follows from (i) and \cite[Theorem 2.2]{Wat83}.

(iii) Since $\init(I)$ is a square-free monomial ideal, $S/\init(I)$ is $F$-injective. By (i), $S/\init(I)$ is Cohen-Macaulay, so we get that $S/I$ is $F$-injective using \cite[Theorem 2.1]{CH97}. So, using (i), we have that $R$ is Cohen-Macaulay, $F$-injective, $F$-rational on the punctured spectrum, and has negative $a$-invariant. In \cite[Theorem 2.8, Remark 1.17]{FW89}, Fedder and Watanabe proved that these conditions, under the stronger assumption that $R$ is an isolated singularity, imply that $R$ is $F$-rational. It turns out that their proof works as well without the stronger assumption (see \cite[Theorem 5.49, Lemma 5.44]{Kol18} for a rigorous proof), hence we conclude.
\end{proof}

\begin{remark}
  \label{rem:ForDRLconjecturesAreStronger}
 Corollary \ref{c:degrevlex} solves in positive Conjectures \ref{conj1} and \ref{conj2} for a degree reverse lexicographic monomial order even. In fact, it uses only the weaker assumption that $\Proj S/I$ is Cohen-Macaulay, instead of  nonsingular. However, this stronger version of Conjectures \ref{conj1} and \ref{conj2} is not true for other monomial orders, as the following examples show.
\end{remark}

\begin{example1}
\begin{enumerate}[\upshape (i)]
\item Let $f=xyz+y^3+z^3\in S=\KK[x,y,z]$, $I=(f)$ and choose the lexicographic monomial order with $x>y>z$. Then $I$ is a homogeneous (w.r.t. the standard grading) prime ideal, $\Proj S/I$ is Cohen-Macaulay and $\init(I)=(xyz)$ is square-free. However $S/\init(I)$ is Cohen-Macaulay with $a$-invariant 0, so Conjecture \ref{conj2} is false if we replace ``$\Proj S/I$ is nonsingular'' by ``$\Proj S/I$ is Cohen-Macaulay''.
\item In \cite[Example 3.4]{Var18}, a homogeneous prime ideal $I$ of the standard graded polynomial ring $S$ in 6 variables over $\KK$ is considered. That ideal satisfies: $\Proj S/I$ is a Cohen-Macaulay surface, $\init(I)$ is square-free for a lexicographic monomial order, but $S/\init(I)$ is not Cohen-Macaulay. So also Conjecture \ref{conj1} is false if we replace ``$\Proj S/I$ is nonsingular'' by ``$\Proj S/I$ is Cohen-Macaulay''.
\end{enumerate}
\end{example1}

In the following corollary, the notion ``graded algebra with straightening laws'' is the same  used in \cite{BV88}.

\begin{corollary}\label{c:asl}
Let $R$ be a graded algebra with straightening laws over a field such that $R$ is a domain. 
\begin{enumerate}[\upshape (i)]
\item If $R$ is generalised Cohen-Macaulay, then $R$ is Cohen-Macaulay with negative $a$-invariant.
\item In characteristic zero, $R$ has rational singularities whenever it has rational singularities on the punctured spectrum.
\item In positive characteristic, $R$ is $F$-rational whenever it is $F$-rational on the punctured spectrum.
\end{enumerate}
\end{corollary}
\begin{proof}
  Let $\Omega$ be the set of ASL generators of $R$, i.e. the partially ordered \KK-algebra generators of $R$ to which the straightening laws apply. This means that  $R\cong S/I$ where $S=\KK[x_{\omega}:\omega\in\Omega]$ is the polynomial ring over $\KK$ whose variables correspond to the elements of $\Omega$
  and $I\subset S$ is a prime ideal such that $\init (I)$ is a square-free monomial ideal with respect to any degrevlex monomial order on $S$ extending the partial order of $\Omega$. By supplying $S$ with the degrees given by $\deg(x_{\omega})=\deg(\omega)$ for all $\omega\in\Omega$, $I$ is a homogeneous ideal of $S$, so we conclude by Corollary \ref{c:degrevlex}.
\end{proof}

\begin{remark}
  \label{rem:isolatedSingularitiesFullfillTheCorollaries}
  If $R=S/I$ is an isolated singularity, the  hypothesis of all three parts of Corollary \ref{c:degrevlex} and Corollary \ref{c:asl} are fulfilled. 
\end{remark}


\section{Gröbner deformations for arbitrary monomial orders}
\label{sec:grobn-deform-arbitr}
From the remaining part of this article we will assume that the graded structure on $S=\KK[x_1,\ldots ,x_n]$ is the standard one.
 \subsection{Complete intersections}
\label{sec:completeIntersections}
 Complete intersections are a particularly nice case to consider: they are always smoothable (in fact there are no obstructions to lifting infinitesimal deformations),
however we are going to see that square-free monomial complete intersections using all the variables of $\KK[x_1,\ldots ,x_n]$ in their minimal generators are not Gröbner-smoothable. The proof of this result is a simple combinatorial argument, that however we find instructive: it is a generalisation of the easy fact that a polynomial defining a Calabi-Yau hypersurface cannot have a square-free initial monomial, which was one of the starting points of this paper.

%

Let $I=(u_1,\ldots ,u_c)$ be a monomial ideal of $S=\KK[x_1,\ldots ,x_n]$ minimally generated by monomials $u_1,\ldots ,u_c$ of degree $d_i=\deg u_i$. 
Then $I$ is a complete intersection (of height $c$) if and only if the $u_i$'s are supported on disjoint sets of variables. So, if $I=I_{\Delta}$ is square-free, then  $d_1+\ldots +d_c\leq n$ with equality holding if and only if $\KK[\Delta]$ has $a$-invariant 0. 

Let $I=I_{\Delta}\subset \mm^2$ be a complete intersection such that $\KK[\Delta]$ has $a$-invariant 0. Then $\Delta$ is the join of $c$ boundaries of simplices. So let $\D=\partial\D_1\ast\dots\ast\partial\D_c$, with $\dim\D_i=d_i-1\ge 1$.
Let the variables corresponding to $\D_i$ be $x_{i1},\dots,x_{id_i}$, so we have
\[
  I_\D=(x_{11}\dots x_{1d_1},\dots,x_{c1}\dots x_{cd_c}).
  \]
  Let $\tau$ be some monomial order, and assume that $I=(f_1,\dots,f_c)$ is a reduced homogeneous Gröbner basis with
  \[
    f_i = x_{i1}\dots x_{id_i}- \sum \la_M M, \textup{where } \la_M\in \KK \textup{~and~}M\le_\tau x_{i1}\dots x_{id_i}.
  \]
  As the Gröbner basis is homogeneous and reduced, all the monomials $M$ above have degree $d_i$ and  are supported on faces of  $\D$. Without losing generality we may always assume that we have the following inequalities for $\tau$:
  \begin{equation}
    \label{eq:orderOfVariablesAssumption}
    \begin{array}{rcccll}
      x_{i1}&>&\dots&>&x_{id_i}&\forall \ i,\\
      x_{11}&\ge&x_{ij}&&&\forall \ x_{ij},\\
      x_{ij}&\ge &x_{cd_c}&&&\forall \ x_{ij}\textup{~with~}i>1.
    \end{array}
  \end{equation}
\begin{lemma}
    \label{lem:ciGBSupport}
    With the above convention, if $\{f_1,\dots,f_c\}$ is a reduced Gröbner basis of the ideal $I$ it generates, then the point $[1:0:\dots:0]$ is a singular point of $\Proj S/I$, where the first projective coordinate corresponds to the variable $x_{11}$.
  \end{lemma}
  \begin{proof}
    First of all, as $x_{11}^a$ is the largest monomial of degree $a$, it does not appear in the support of any $f_i$, so $[1:0:\dots:0]\in \Proj S/I$. Assume the contrary, namely that $[1:0:\dots:0]$ is a smooth point of $\Proj S/I$. Then the Jacobian computed at $[1:0:\dots:0]$ has to have maximal rank $c$. This implies that  no row can be zero when substituting. In particular,  in the support of $f_c$, we must have some monomial of the form $x_{11}^{d_c-1}x_{ij}$. If $i>1$, then $x_{cd_c}< x_{ij}$. Because all other variables $x_{ck}<x_{11}$, we have $x_{c1}\dots x_{cd_c} < x_{11}^{d_c-1}x_{ij}$, contradicting $\init(f_c)=x_{c1}\dots x_{cd_c}$. So the monomial has to be $x_{11}^{d_c-1}x_{1j}$ for some $j$. This implies two things:
    \begin{eqnarray}
      \label{eq:d1>2}
      x_{11}x_{1j}\in \D&&\textup{because the Gröbner basis is reduced}.\\
      \label{eq:1j<cdc}
      x_{1j}<x_{cd_c}&& \textup{because~}  x_{c1}\dots x_{cd_c} > x_{11}^{d_c-1}x_{1j}.                             
    \end{eqnarray}
    From \eqref{eq:d1>2} we obtain that $d_1>2$.  Given that the row of the Jacobian corresponding to $f_1$ cannot be zero when substituting $[1:0:\dots:0]$, we must have some monomial $x_{11}^{d_1-1}x_{kl}$ in the  support of $f_1$. We claim that this leads to a contradiction. There are two cases:\\
    \noindent    If $k=1$, then $x_{11}^{d_1-1}x_{1l}>x_{11}\dots x_{1d_1}$ which is  a contradiction to $\init(f_1)=x_{11}\dots x_{1d_1}$.\\
    \noindent    If $k>1$, then by \eqref{eq:1j<cdc} we have $x_{1j}<x_{cd_c}$, and by assumption (\ref{eq:orderOfVariablesAssumption}) we have $x_{cd_c}<x_{kl}$. So
    $x_{11}\dots x_{1d_1}<x_{11}^{d_1-1}x_{1j}<x_{11}^{d_1-1}x_{kl}$ which is again a contradiction to $\init(f_1)=x_{11}\dots x_{1d_1}$.
  \end{proof}
  
  \begin{corollary}
    \label{cor:ciAreNeverGrobSmoothable}
A join of boundaries of positive dimensional simplices is not Gröbner-smoothable over any field $\KK$.
\end{corollary}

The simplicial complexes of the above corollary are particular combinatorial spheres. If Conjecture \ref{conj3} is true, no combinatorial sphere is Gröbner-smoothable. As a consequence of the results in the next subsection, this is at least true for 1-dimensional spheres. 
  
\subsection{Dimension 1}
\label{sec:ManifoldsDimnesion1}

In Proposition \ref{p:equivalent} we proved that if Conjecture \ref{conj1} is true in dimension $\leq d+1$, then Conjecture \ref{conj2} is true in dimension $\leq d$. On the other hand Conjecture \ref{conj1} is true whenever $\dim S/I \leq 2$ by Proposition \ref{p:necessary} (since a 1-dimensional simplicial complex is Cohen-Macaulay if and only if it is strongly connected), so the first instance where it is open is when $\dim S/I=3$. This justifies our next interest, that is to study Conjecture \ref{conj2} when $\dim S/I=2$. In this case, Conjecture \ref{conj2} says that if $I\subset S$ defines a smooth projective curve $C$ and $\init(I)$ is square-free for some monomial order, then the genus of $C$ must be zero. While the most general case remains open, in this section we provide some evidence for it.

\subsubsection{Elliptic curves over real number fields}
In the proof of the next theorem we use the result of Elkies \cite{Elk87} which states that an elliptic curve over the rational numbers has infinitely many supersingular primes.

\begin{theorem}
  \label{t:elliptic}
If $I\subset \mathbb{Q}[x_1,\ldots ,x_n]$ defines a nonsingular projective curve of genus one, then $\init(I)$ is not square-free for any monomial order.
\end{theorem}
\begin{proof}
Given a prime number $p$, we will denote by $S_p$ the polynomial ring $\ZZ/p\ZZ[x_1,\ldots ,x_n]$. Also, for an ideal $J\subset S=\mathbb{Q}[x_1,\ldots ,x_n]$ we write $J_p$ for the ideal $J'S_p\subset S_p$, where $J'=J\cap \ZZ[x_1,\ldots ,x_n]$.

Let $E\subset \PP^{n-1}$ be the nonsingular projective curve of genus 1 defined by $I\subset S$. Of course $I_p$ defines $E_p$, its reduction mod $p$, for all prime numbers $p$. Since any nonsingular projective curve of genus one can be embedded as an elliptic curve in $\PP^2$, there exist infinitely many supersingular primes $p$ for $E$ by \cite[Theorem 1]{Elk87}. This means that the Frobenius does not act injectively on $H^1(E_p,\mathcal{O}_{E_p})$ for infinitely many primes $p$. Since $H^1(E_p,\mathcal{O}_{E_p})\cong H_{\mm_p}^2(S_p/I_p)_0$, this implies
\begin{equation}
  \label{eq:ElkiesConclusionOnFinjective}
  \textup{$S_p/I_p$ is not $F$-injective for infinitely many primes $p$.}
\end{equation}

Suppose by contradiction that $\init(I)$ is square-free for some monomial order. So $\init(I)=I_{\Delta}$ for some 1-dimensional simplicial complex $\Delta$. By Proposition \ref{p:necessary} (i) $\Delta$ must be a connected 1-dimensional simplicial complex, thus $\D$ is shellable, and thus $S/I_{\Delta}$ is Cohen-Macaulay. Looking at the Buchberger algorithm, it is easy to realise that $\init(I_p)=(I_{\Delta})_p$ for all prime numbers $p\gg 0$. So $S_p/\init(I_p)$ is $F$-injective and Cohen-Macaulay for all $p\gg 0$; therefore $S_p/I_p$ is $F$-injective by \cite[Theorem 2.1]{CH97} for all $p\gg 0$ as well, a contradiction to (\ref{eq:ElkiesConclusionOnFinjective}).
 
\end{proof}

\begin{remark}
  \label{rem:extendToNumberfields}
With the appropriate notion of reduction mod $p$, the proof of Theorem \ref{t:elliptic} works replacing $\mathbb{Q}$ with any number field $\KK\subset \mathbb{R}$ by \cite{Elk89}.
The proof fails however on a general field of characteristic zero. For example, all the reductions mod $p$ of the ring $\mathbb{Q}(\pi)[x,y,z]/(x^3+y^3+z^3+\pi\cdot xyz)$ are $F$-injective for all prime numbers $p$, thus (\ref{eq:ElkiesConclusionOnFinjective}) fails. 
\end{remark}

\subsubsection{Leafless 1-dimensional  complexes}

Theorem \ref{t:elliptic} together with Remark \ref{rem:extendToNumberfields}  say that a graph with one cycle is not Gröbner-smoothable over a real number field.
We extend this to arbitrary fields and any number of cycles, with the extra assumption that the simplicial complex involved has no leafs. By a leaf we mean a vertex which is contained in exactly one maximal face. We call a simplicial complex \emph{leafless} if it has no leafs. 

\begin{remark}
  Recall that a \emph{free face} of a pure $d$-dimensional simplicial complex is a face of dimension $d-1$ contained in only one maximal face. The notion of pure simplicial complex with no free faces coincides with the notion of leafless pure simplicial complex only in dimension 1. In higher dimension, leafless  is implied by the absence of free faces.

\end{remark}

We start with the setup. 
Since $I\subset S=\KK[x_1,\ldots, x_n]$ is a homogeneous prime ideal such that $\dim S/I=2$ and $\init(I)=I_{\D}$ for some monomial order, then by Proposition \ref{p:necessary} one has that  $\D$ is a 1-dimensional, connected simplicial complex on $[n]$, i.e., $\D$ is a connected graph on $n$ vertices. Moreover, from now on we will assume that $I$ does not contain linear forms (harmless for our goals), that is,  $\{i\}\in\D$ for every $i\in [n]$.  Let  $\{f_F~|~F \textup{~is a minimal non-face of~}\D\}$ be a homogeneous reduced Gröbner basis of $I_{\Delta}$ with respect to some monomial order $\tau$, such that 
 \begin{equation}
   \init_\tau f_{F} = \bfx_F\textup{~for every minimal~}F\notin\D.
 \end{equation}
 We also use the notation
 \[
   d_F:=\deg f_F \textup{~for every minimal~}F\notin\D.
 \]
 As $\D$ is 1-dimensional, $d_F$ is either two or three. We may and will relabel the vertices of $\D$ in such a way that $\link_\D 1 = \{2,\dots,r\}$ (we abuse here notation and write just $2$ for the face $\{2\}$ and so on). For the fixed monomial order $\tau$  we may always assume without loss of generality that we have the following inequalities:
 \begin{equation}
   \label{eq:orderOfVariablesAssumption2}
  \begin{array}{rclll}
    x_{1}&>&x_{i}&\forall \ i>1,\\
    x_{2}&>&\dots>&x_r.&\\
  \end{array}
\end{equation}
\begin{remark}
  \label{claim1}
  If $\{1,l\}\notin\D$, then  $x_1^{d_F-1}x_{l}\notin\supp{f_F}$ and $x_1^{d_F}\notin\supp{f_F}$ for every minimal non-face $F$. 
\end{remark}
\begin{proof}
  This follows directly from the fact that $x_1$ is the highest variable, and the Gröbner basis is reduced.
\end{proof}

\begin{lemma}
   \label{lem:x1x2notInSupport}
   For $1<a\le \min\{r,3\}$, we have that
   \begin{equation}
     \label{eq:x1x2}
     x_1^{d_F-1}x_a\notin\supp f_{F}, \textup{~for all minimal non-faces with~} 1\notin F.
   \end{equation}
 \end{lemma}

 \begin{proof}
   We label the minimal non-faces of $\D$ not containing $1$ by $F_1,\dots,F_s$, denote the corresponding degrees by $d_1,\dots,d_s$, and assume that
   \begin{equation}
     \label{eq:inductiveOrderFi}
     \textup{if~} 1\le i<j\le s,\textup{~then~} x_1^{3-d_i}\init_\tau f_{F_i}< x_1^{3-d_j}\init_\tau f_{F_j}.      
   \end{equation}
   We will prove \eqref{eq:x1x2} inductively, that is, we assume that \eqref{eq:x1x2} holds for $F_1,\dots,F_{k-1}$ and prove that it also holds for $F_k$.
   
   Let $a=2$ and assume that $x_1^{d_k-1}x_2\in\supp f_{F_k}$. This implies that $x_1^{d_k-1}x_2<\bfx_F$, so that $x_2< x_i$ for all $i\in F_k$. Since $x_2< x_i$ implies $i>r$ by (\ref{eq:orderOfVariablesAssumption2}), it follows that $\{1,i\}\notin\D$ for every $i\in F_k$. Let $i_k$ be one of the vertices in $F_k$ and build the S-polynomial
   \[
     S(f_{\{1,i_k\}},f_{F_k}) = \bfx_{F_k\setminus i_k} f_{\{1,i_k\}} - x_1f_{F_k}.
   \]
   As we are dealing with a Gröbner basis, we must have
   \begin{equation}
     \label{eq:rewriteGB}
     S(f_{\{1,i_k\}},f_{F_k}) = \sum g_F\cdot f_F,\textup{~with~}\init (g_F\cdot f_F) < x_1\bfx_{F_k}.
   \end{equation}
   The highest exponent of $x_1$ in a monomial in the support of $\bfx_{F_k\setminus i_k}f_{\{1,i_k\}}$ is 1. So, because $x_1^{d_k}x_2\in\supp x_1f_{F_k}$ with $d_k>1$,  we have $x_1^{d_k}x_2\in\supp  S(f_{\{1,i_k\}},f_{F_k})$. Thus, there must appear  some non-face $\alF$ on the right hand side of (\ref{eq:rewriteGB}), such that
   \begin{equation}
     \label{eq:x1x2inSuppGF}
     x_1^{d_k}x_2\in\supp g_{\alF}f_{\alF}.     
   \end{equation}
   \begin{claim}
     \label{claim2}
     For every $\alF$ as above we have $1\notin \alF$.  
   \end{claim}
   \noindent{\it Proof of Claim \ref{claim2}:}
We assume that $1\in\alF$ and seek a contradiction. There are two cases:\\
   If $|\alF| = 3$, then let $\alF=\{1,l,m\}$, so $l,m\in\link_\D1$. In particular $x_{l}\le x_2$ and $x_{m}\le x_2$, so $x_1^2x_2>x_1x_{l}x_{m} = \init(f_{\alF})$. This implies $\init(g_{\alF}f_{\alF})< x_1^{d_k}x_2$, so $x_1^{d_k}x_2\notin\supp g_{\alF}f_{\alF}$, which is a contradiction to (\ref{eq:x1x2inSuppGF}). \\
   If $|\alF|=2$, then let $\alF=\{1,l\}$. Because $x_1^2\notin\supp f_{\alF}$ we must have by (\ref{eq:x1x2inSuppGF}) that $x_1^{d_k-1}\in\supp g_{\alF}$, thus also that $x_1^{d_k}x_{l}\in\supp g_{\alF}f_{\alF}$. 
   By Remark \ref{claim1}, we have that  $x_1^{d_k}x_{l}\notin\supp  S(f_{\{1,i_k\}},f_{F_k})$,  so $x_1^{d_k}x_{l}$ must cancel out with some monomial in $\supp(g_Ff_F)$ for some $F$. But, again by Remark \ref{claim1}, we have $x_1^{d_k}x_l\notin\supp(g_Ff_F)$ for every $F$ and we also obtain a contradiction. \hfill {\tiny $_\blacksquare$} \\

    By Claim \ref{claim2} we have $\alF=F_i$ for some $i\in\{1,\dots,s\}$. 
    As $x_1^{d_{\alF}}\notin\supp f_{\alF}$, we must have $x_1^{d_{\alF}-1}x_2\in\supp f_{\alF}$. 
    By the inductive hypothesis we must have that $i\ge k$.
    From (\ref{eq:inductiveOrderFi}) it follows that $x_1^{3-d_0}\init f_{\alF}\ge x_1^{3-d_k}\init f_{F_k}$.
    Thus we get  $\init_\tau g_{\alF}f_{\alF}\ge x_1\bfx_{F_k}$, which is a contradiction to \eqref{eq:rewriteGB}.

   When $a=3$, the only change in the proof is in the first lines:  From $x_3<x_i$ for $i\in F_k$ it follows that there exists $j\in F_k$ with $\{1,j\}\notin\D$ (for $x_2$ we had this for any $i\in F_k$). The rest of the proof runs analogously.
 \end{proof}

 The proof of Lemma \ref{lem:x1x2notInSupport} no longer works  for the third-largest variable in the link of $1$, namely $x_1x_4 \in\supp f_{\{2,3\}}$ would not lead to a similar contradiction. A straight-forward generalisation of the argument in higher dimension is also not possible, because the link of 1 is no longer a 0-dimensional complex, so we may no longer assume that $x_2>\dots >x_r$. So Lemma \ref{lem:x1x2notInSupport} is the best possible result using this type of arguments related to monomial orders and reduced Gröbner bases. 
 \begin{remark}
   \label{rem:deg3gens}
   For a minimal non-face $F$ with $|F|=3$ and $1\in F$, we have
   \[
     x_1^2x_2,x_1^2x_3\notin \supp f_F.
   \]
 \end{remark}
 \begin{proof}
    Because $F=\{1,l,m\}$ is a minimal non-face, we have $l,m\in\link_\D1$. As we chose $x_2$ and $x_3$ to be the largest two variables in the link of $1$, we have $x_2\ge x_l$ and $x_3\ge x_m$ (or $l$ and $m$ switched). Thus $x_1^2x_2>x_1^2x_3>x_1x_kx_l$.
 \end{proof}

\begin{theorem}
  \label{thm:noGBsmoothingLocDeg>1}
  Let $\D$ be a 1-dimensional simplicial complex, and fix a monomial order.
Let $I\subset S$ be homogeneous ideal with $\init I = I_\D$, and let $\{f_F~|~F \textup{~is a minimal non-face of~}\D\}$ be a reduced homogeneous Gröbner basis of $I$. If $1$ is not a leaf, then  $P_1=[1:0:\dots:0]$ is a singular point of $\Proj S/I$.
\end{theorem}
\begin{proof}
  Clearly $x_1^{d_F}$ is not in the support of any $f_F$, which implies that $P_1 \in \Proj S/I$.   Because $1$ is not a leaf we have $r\ge 3$, where $\{2,\dots,r\}=\link_\D 1$. The Jacobian $\Jac(I)$ can be split in two blocks with full rows as follows:
  \[
  \begin{array}{l}
     B_1 := \textup{ the~} n-r \textup{~rows corresponding to the degree 2 generators } f_F \textup{~with }x_1|\bfx_F.  \\
     B_2 := \textup{ the rows corresponding to the ~}f_F\textup{~ without~}x_1^{d_F-1}x_2 \textup{~and~}x_1^{d_F-1}x_3\textup{~in their support.}
  \end{array}
  \]
  When substituting $P_1$ in the Jacobian,  the columns indexed by $x_{r+1},\dots,x_n$ produce an identity submatrix in $B_1(P_1)$, so 
  \[
    \rank B_1(P_1)=n-r.
  \]
  In $B_2(P_1)$, the columns indexed by $x_{r+1},\dots,x_n$ are all zero because the Gröbner basis is reduced. The first column is zero because $x_1^{d_F}\notin\supp f_F$ for any minimal non-face $F$. By  Lemma \ref{lem:x1x2notInSupport} and Remark \ref{rem:deg3gens} the columns indexed by $x_2$ and $x_3$ are also zero in $B_2(P_1)$. This leaves at most $r-3$ nonzero columns in $B_2(P_1)$, thus
   \[
    \rank B_2(P_1)\le r-3.
  \]
  This means that $\rank\Jac(P_1)\le n-r + r-3 = n-3 < \codim_{\PP^{n-1}} \Proj S/I$, so $P_1$ is a singular point.
\end{proof}

\begin{corollary}
  \label{cor:leaflessGraphsNotGRSmoothable}
A graph which is Gröbner-smoothable over some field must have leafs. In particular, a cycle is not Gröbner-smoothable over any field.
\end{corollary}

\begin{remark}\label{r:lex}
For lexicographic monomial orders the above arguments are much easier and work in any dimension: If $\tau$ is the lexicographic order corresponding to $x_1>\dots>x_n$, we have that if $x_1| M$ with $M\in\supp f_F$, then $x_1| \bfx_F$. Furthermore, $x_1^2$ does not divide any monomial in the support of any $f_F$. 
So $(\frac{\partial f_F}{\partial x_i}(P_1))_{i=1,\dots,n} \neq (0,\dots,0)$ if and only if $F=\{1,j\}$ where $\{1,j\}\notin \Delta$. This means $P_1=[1:0,\ldots ,0]$ can be a smooth point of $\Proj(S/I)$ only if $|\link_{\Delta}1|\leq \dim \Delta$. This leads to the next result.
\end{remark}

\begin{proposition}
  \label{prop:noGroebnerLexSmoothings}
 In any dimension, a leafless simplicial complex $\Delta$  is not Gröbner-smoothable with respect to any lexicographic order. In particular, pseudo-manifolds are not lexicographically Gröbner-smoothable.
\end{proposition}
\begin{proof}
Since $I_\Delta$ is the initial ideal of a prime ideal, it must be strongly connected by Proposition \ref{p:necessary} (i), and thus pure. So Remark \ref{r:lex} yields the result.
\end{proof}

\bibliographystyle{amsalpha}
\bibliography{Alexbib}%
\vfill

\end{document}